\documentclass{article}
\usepackage{graphicx}
\usepackage{amsfonts}
\usepackage{amsmath}
\usepackage{amsthm}
\usepackage{harvard}
\usepackage{xcolor}
\usepackage{overpic}
\usepackage{algorithm}
\usepackage[noend]{algpseudocode}

\newtheorem{theorem}{Theorem}[section]
\newtheorem{lemma}[theorem]{Lemma}

\newtheorem{proposition}[theorem]{Proposition}
\theoremstyle{definition}

\theoremstyle{remark}
\newtheorem{remark}[theorem]{Remark}

\allowdisplaybreaks

\newcommand{\ee}{{\mathrm e}}
\newcommand{\ii}{{\mathrm i}}

\newcommand{\hyper}[5]
       {{}_{#1} F_{#2} \!\left[
           \begin{array}{l}
         #3;\\#4;
           \end{array}#5\right]
       }
\newcommand{\LL}{\mathrm{L}}

\newcommand{\BB}[1]{\mathbb{#1}}
\newcommand{\CC}[1]{\mathrm{#1}}
\renewcommand{\O}[1]{\mathcal{O}\left( #1 \right)}

\newcommand{\R}[1]{\eqref{#1}}
\newcommand{\D}{\mathrm{d}}
\renewcommand{\ee}{\mathrm{e}}
\renewcommand{\ii}{\mathrm{i}}

\begin{document}
\title{A family of orthogonal rational functions and other orthogonal systems with a skew-Hermitian differentiation matrix}
\author{Arieh Iserles\\
Department of Applied Mathematics and Theoretical Physics\\
Centre for Mathematical Sciences\\
University of Cambridge\\
Wilberforce Rd, Cambridge CB4 1LE\\
United Kingdom
 \and 
 Marcus Webb\\
 Department of Mathematics\\
 University of Manchester\\
 Alan Turing Building, Manchester M13 9PL\\
 United Kingdom}
 
\thispagestyle{empty}
\maketitle

\begin{abstract}
  In this paper we explore orthogonal systems in $\CC{L}_2(\BB{R})$ which give rise to a skew-Hermitian, tridiagonal differentiation matrix. Surprisingly, allowing the differentiation matrix to be complex leads to a particular family of rational orthogonal functions with favourable properties: they form an orthonormal basis for $\CC{L}_2(\BB{R})$,  have a simple explicit formulae as rational functions, can be manipulated easily and the expansion coefficients are equal to classical Fourier coefficients of a modified function, hence can be calculated rapidly. We show that this family of functions is essentially the only orthonormal basis possessing  a differentiation matrix of the above form and whose coefficients are equal to classical Fourier coefficients of a modified function though a monotone, differentiable change of variables. 
  Examples of other orthogonal bases with skew-Hermitian, tridiagonal differentiation matrices are discussed as well.
\end{abstract}

\noindent \textbf{Keywords} \, Orthogonal systems, orthogonal rational functions, spectral methods, Fast Fourier Transform, malmquist-Takenaka system

\noindent \textbf{AMS classification numbers} \, Primary: 41A20, Secondary: 42A16, 65M70, 65T50

\section{Introduction}

The motivation for this paper is the numerical solution of time-dependent partial differential equations on the real line. It continues an ongoing project of the present authors, begun in \cite{iserles18osssd}, which studied orthonormal systems $\Phi = \{\varphi_n\}_{n\in\BB{Z}}$ in $\CC{L}_2(\BB{R})$ which satisfy the differential-difference relation,
\begin{equation}
  \label{eq:skew}
  \varphi_n'(x) = -b_{n-1} \varphi_{n-1}(x) + b_n\varphi_{n+1}(x), \qquad n \in \BB{Z}_+,
\end{equation}
for some real, nonzero numbers $\{b_n\}_{n\in\BB{Z}}$ where $b_{-1} = 0$.  In other words, the {\em differentiation matrix\/} of $\Phi$ is skew-symmetric, tridiagonal and irreducible. The virtues of skew symmetry in this context are elaborated in \cite{hairer16nsp,iserles16jps} and \cite{iserles18osssd} -- essentially, once $\Phi$ has this feature, spectral methods based upon it typically allow for a simple proof of numerical stability and for the conservation of energy whenever the latter is warranted by the underlying PDE. The importance of tridiagonality is clear, since tridiagonal matrices lend themselves to simpler and cheaper numerical algebra. 

In this paper we generalise \R{eq:skew}, allowing for a {\em skew-Hermitian\/} differentiation matrix. In other words, we consider systems $\Phi$ of complex-valued functions such that
\begin{equation}
  \label{eq:sHerm}
  \varphi_n'(x) = -\overline{b}_{n-1}\varphi_{n-1}(x) + \ii c_n\varphi_n(x) + b_n\varphi_{n+1}(x),
\end{equation}
where $\{b_n\}_{n\in \BB{Z}_+} \subset \BB{C}$ and $\{c_n\}_{n\in\BB{Z}_+} \subset \BB{R}$. 

While the substantive theory underlying the characterisation of orthonormal systems in $\CC{L}_2(\BB{R})$ with skew-Hermitian, tridiagonal, irreducible differentiation matrices is a fairly straightforward extension of \cite{iserles18osssd}, its ramifications are new and, we believe, important. In Section~2 we establish this theory, characterising $\Phi$ as Fourier transforms of weighted orthogonal polynomials with respect to some absolutely-continuous Borel measure $\D\mu$. This connection is reminiscent of \cite{iserles18osssd} but an important difference is that $\D\mu$ need not be symmetric with respect to the origin: this affords us an opportunity to consider substantially greater set of candidate measures. 

An important issue is that, while the correspondence with Borel measures guarantees orthogonality and the satisfaction of \R{eq:sHerm}, it does not guarantee completeness. In general, once $\D\mu$ is determinate and supported by the interval $(a,b)$, completeness is assured in the {\em Paley--Wiener space\/} $\mathcal{PW}_{(a,b)}(\BB{R})$. 

So far, the material of this paper represents a fairly obvious generalisation of \cite{iserles18osssd}. Furthermore, the operation of differentiation for functions on the real line is defined without venturing into the complex plane. Indeed, it is legitimate to challenge why we should allow our differentiation matrices to contain complex numbers. After all, if skew-Hermitian framework is so similar to the (simpler!) skew-symmetric one, why bother? The only possible justification is were \R{eq:sHerm} to confer an advantage (in particular, from the standpoint of computational mathematics) in comparison with \R{eq:skew}. This challenge is answered in Section~3 , where we consider sets $\Phi$ associated with generalised Laguerre polynomials, where $(a,b)=(0,\infty)$. We show that a simple tweak to our setting assures the completeness of these {\em Fourier--Laguerre functions,\/} which need be indexed over $\BB{Z}$, rather than $\BB{Z}_+$.

The Fourier--Laguerre functions in their full generality, while expressible in terms of the Szeg\H{o}--Askey polynomials on the unit circle, are fairly complicated. However, in the case of the simple Laguerre measure $\D\mu(x)=\chi_{(0,\infty)}(x)\ee^{-x}\D x$ they reduce to the {\em Malmquist--Takenaka (MT) system\/}
\begin{equation}
  \label{MT}
  \varphi_n(x)=\sqrt{\frac{2}{\pi}}\ii^n \frac{(1+2\ii x)^n}{(1-2\ii x)^{n+1}},\qquad n\in\BB{Z}.
\end{equation}
The MT system has been discovered independently by \citeasnoun{malmquist26stc} and \citeasnoun{takenaka25oof} and investigated by many mathematicians, in different contexts: approximation theory \cite{bultheel03fat,bultheel99orf,higgins2004completeness,weideman95tao}, harmonic analysis \cite{eisner14dom,pap15ecm}, signal processing \cite{wiener1949extrapolation} and spectral methods \cite{christov1982complete}. Some of these references are aware of the original work of Malmquist and Takenaka, while others reinvent the construct.


A remarkable property of the MT system \R{MT} is that the computation of the expansion coefficients
\begin{displaymath}
  \hat{f}_n=\int_{-\infty}^\infty f(x)\varphi_n(x)\D x,\qquad n\in\BB{Z},
\end{displaymath}
can be reduced, by an easy change of variables, to a standard Fourier integral. Therefore the evaluation of $\hat{f}_{-N},\ldots,\hat{f}_{N-1}$ can be accomplished with the Fast Fourier Transform (FFT) in $\O{N\log_2N}$ operations: this has been already recognised, e.g.\ in \cite{weideman95tao}. In Section~4 we characterise all systems $\Phi$, indexed over $\BB{Z}$, which tick all of the following boxes:
\begin{itemize}
\item They are orthonormal and complete in $\CC{L}_2(\BB{R})$,
\item They have a skew-Hermitian, tridiagonal differentiation matrix, and
\item Their expansion coefficients $\hat{f}_{-N},\ldots,\hat{f}_{N-1}$ can be approximated with a discrete Fourier transform by a single change of variables, and hence computed in $\O{N\log_2N}$ operations with fast Fourier transform.
\end{itemize}
Adding rigorous but reasonable assumptions to these requirements, we prove that, modulo a simple generalisation, the MT system is the \emph{only} system which bears all three. 

We wish to draw attention to \cite{iserles19fao}, a companion paper to this one. While operating there within the original framework of \cite{iserles18osssd} -- skew-symmetry rather than skew-Hermicity -- we seek therein to characterise orthonormal systems in $\CC{L}_2(\BB{R})$ whose first $N$ coefficients can be computed in $\O{N\log_2N}$ operations by fast expansion in orthogonal polynomials. We identify there a number of such systems, all of which can be computed by a mixture of fast cosine and fast sine transforms. Such systems are direct competitors to the Malmquist--Takenaka system, discussed in this paper.

\setcounter{equation}{0}
\section{Orthogonal systems with a skew-Hermitian differentiation matrix}

\subsection{Skew-Hermite differentiation matrices and Fourier transforms}

The subject matter of this section is the determination of verifiable conditions equivalent to the existence of a skew-Hermitian, tridiagonal, irreducible differentiation matrix \R{eq:sHerm} for a system $\Phi=\{\varphi_n\}_{n\in\BB{Z}_+}$ which is orthonormal in $\CC{L}_2(\BB{R})$.

\begin{theorem}[Fourier characterisation for $\Phi$]
  \label{thm:nonorthogonalPhi}
  The set  $\Phi = \{\varphi_n\}_{n\in\BB{Z}_+}\subset\CC{L}_2(\BB{R})$ has a skew-Hermitian, tridiagonal, irreducible differentiation matrix \R{eq:sHerm} if and only if
  \begin{equation}
  \label{explicit_phi}
  \varphi_n(x) = \frac{\ee^{\ii \theta_n}}{\sqrt{2\pi}} \int_{-\infty}^\infty \ee^{\ii x \xi} p_n(\xi) g(\xi) \, \D \xi,
  \end{equation}
  where $P = \{p_n\}_{n\in\BB{Z}_+}$ is an orthonormal polynomial system on the real line with respect to a non-atomic probability measure $\mathrm{d}\mu$ with all finite moments\footnote{By this, we mean that $\mathrm{d}\mu$ is Borel measure on the real line with total mass equal to 1 and with an uncountable number of points of increase (for example all of $\BB{R}$ or the interval $[-1,1]$).}, $g$ is a square-integrable function which decays superalgebraically fast as $|\xi| \to \infty$, and $\{\theta_n\}_{n\in\BB{Z}_+}$ is a sequence of numbers in $[0,2\pi)$. Furthermore, $P$, $g$, and $\{\theta_n\}_{n\in\BB{Z}_+}$ are uniquely determined by $\varphi_0$, $\{c_n\}_{n\in\BB{Z}_+}$, and $\{b_n\}_{n\in\BB{Z}_+}$\footnote{We assume by convention that the leading coefficients of the elements of $P$ are positive.}.
\end{theorem}

\begin{remark}
This theorem is a straightforward generalisation of \cite[Thm.~6]{iserles18osssd}, which shows the same result but for real, irreducible skew-symmetric differentiation matrices. The difference is that  \R{eq:sHerm}  is replaced by \R{eq:skew}, $\mathrm{d}\mu$ must be even, $g$ must have even real part and odd imaginary part, and $\theta_n$ is chosen so that $\ee^{\ii \theta_n} = (-\ii)^n$. We will prove sufficiency because it is elementary but enlightening, and leave necessity and uniqueness for the reader to prove by modifying the proof in \cite{iserles18osssd}. That part of the proof depends on Favard's theorem and properties of the Fourier transform, and we wish  to avoid it for the sake of brevity.
\end{remark}

\begin{proof} Suppose that $\varphi_n$ are given by the equation \eqref{explicit_phi}. Then by \cite[Thm.~1.29]{gautschi2004orthogonal} there exist real numbers $\{\delta_n\}_{n \in\BB{Z}_+}$ and positive numbers $\{\beta_n\}_{n\in\BB{Z}_+}$ such that
  \begin{equation}
   \label{TTRR}
  \xi p_n(\xi) = \beta_{n-1} p_{n-1}(\xi) + \delta_n p_n(\xi) + \beta_n p_{n+1}(\xi), \qquad n \in \BB{Z}_+,
  \end{equation}
  where $\beta_{-1} = 0$ by convention.\footnote{This form \R{TTRR} of the three-term recurrence relation for $P$ ensures orthonormality of the underlying orthogonal polynomials.} Differentiating under the integral sign and using the above three-term recurrence, we obtain
  \begin{displaymath}
  \varphi'_n(x) = \ii \ee^{\ii(\theta_n - \theta_{n-1})} \beta_{n-1} \varphi_{n-1}(x) + \ii\delta_n \varphi_n(x) + \ii \ee^{\ii(\theta_n - \theta_{n+1})}\beta_n \varphi_{n+1}(x).
  \end{displaymath}
  Set $c_n = \delta_n$ and $b_n = \ii \ee^{\ii(\theta_n - \theta_{n+1})}\beta_n$ for $n \in \BB{Z}_+$. Then $c_n \in \BB{R}$ and $-\overline{b}_{n-1} = -(-\ii)\ee^{\ii(\theta_n - \theta_{n-1})}\beta_{n-1} = \ii \ee^{\ii(\theta_n - \theta_{n-1})} \beta_{n-1}$, so that $\Phi$ satisfies equation \eqref{eq:sHerm}.
  \end{proof}

Theorems \ref{thm:orthogonalPhi} and \ref{thm:PW} are proved in \cite{iserles18osssd} for the real case, as in equation \eqref{eq:skew}. The proofs require minimal modification for them to apply to the complex case, as in equation \eqref{eq:sHerm}.

\begin{theorem}[Orthogonal systems]\label{thm:orthogonalPhi}
  Let $\Phi = \{\varphi_n\}_{n\in\BB{Z}_+}$ satisfy the requirements of Theorem \ref{thm:nonorthogonalPhi}. Then $\Phi$ is orthogonal in $\mathrm{L}_2(\BB{R})$ if and only if $P$ is orthogonal with respect to the measure $|g(\xi)|^2\mathrm{d}\xi$. Furthermore, whenever $\Phi$ is orthogonal, the functions $\varphi_n / \|g\|_{2}$ are orthonormal.
\end{theorem}

\begin{theorem}[Orthogonal bases for a Paley--Wiener space]\label{thm:PW}
  Let $\Phi = \{\varphi_n\}_{n\in\BB{Z}_+}$ satisfy the requirements of Theorem \ref{thm:orthogonalPhi} with a measure $\mathrm{d}\mu$ such that polynomials are dense in $\CC{L}_2(\BB{R};\mathrm{d}\mu)$. Then $\Phi$ forms an orthogonal basis for the Paley--Wiener space $\mathcal{PW}_\Omega(\BB{R})$, where $\Omega$ is the support of $\mathrm{d}\mu$.
\end{theorem}

The key corollary of Theorem \ref{thm:PW} is that for a basis $\Phi$ satisfying the requirements of Theorem \ref{thm:orthogonalPhi} to be complete in $\CC{L}_2(\BB{R})$, it is necessary that the polynomial basis $P$ is orthogonal with respect to a measure which is supported on the whole real line.

\subsection{Symmetries and the canonical form}\label{subsec:symmetries}


Let $\Phi = \{\varphi_n\}_{n\in\BB{Z}_+}$ have a tridiagonal skew-Hermitian differentiation matrix as in equation \eqref{eq:sHerm}. Then the system $\tilde\Phi = \{\tilde\varphi_n\}_{n\in\BB{Z}_+}$ defined by
  \begin{equation}
  \label{Affine}
  \tilde{\varphi}_n(x) = A\ee^{\ii (\omega x + \kappa_n)} \varphi_n(B x + C),
  \end{equation}
  where $\omega, A, B, C, \kappa_n \in \BB{R}$ and $A,B \neq 0$, also satisfies equation \eqref{eq:sHerm}. We can show this directly as follows.
  \begin{eqnarray*}
     \tilde{\varphi}'_n(x) &=& AB\ee^{\ii (\omega x + \kappa_n)} \varphi'_n(B x + C) + A\ii\omega\ee^{\ii (\omega x + \kappa_n)} \varphi_n(B x + C) \\
     &=& AB\ee^{\ii (\omega x + \kappa_n)}[-\overline{b}_{n-1}\varphi_{n-1}(Bx+C) + \ii c_n\varphi_n(Bx+C) + b_n\varphi_{n+1}(Bx+C)] \\
     & & \qquad + \ii\omega A\ee^{\ii (\omega x + \kappa_n)} \varphi_n(B x + C) \\
     &=& -B\ee^{\ii(\kappa_n - \kappa_{n-1})}\overline{b}_{n-1} \tilde\varphi_{n-1}(x) + \ii (Bc_n + \omega)\tilde\varphi_n(x) + B\ee^{\ii(\kappa_n - \kappa_{n+1})}b_{n} \tilde\varphi_{n+1}(x) \\
     &=& -\overline{b}^\circ_{n-1} \tilde\varphi_{n-1}(x) + \ii c^\circ_n \tilde{\varphi}_n(x) + b^\circ_n\tilde\varphi_{n+1}(x),
    \end{eqnarray*}
  where $b^\circ_n = B\ee^{\ii(\kappa_n - \kappa_{n+1})}b_{n}$ and $c^\circ_n = Bc_n + \omega$.
  
  The parameters $\omega, A, B, C, \kappa_0,\kappa_1,\kappa_2,\ldots$ encode continuous symmetries in the space of systems with a tridiagonal skew-Hermitian differentiation matrix. Note that these symmetries also preserve orthogonality (but not necessarily orthonormality). 
  
  If the differentiation matrix is irreducible then these symmetries permit a unique choice of $\kappa_0,\kappa_1,\ldots$ ensuring that $b_n$ is a positive real number for each $n \in \BB{Z}_+$. This corresponds to modifying the choice of $\theta_n$ in Theorem \ref{thm:nonorthogonalPhi} so that $\ee^{\ii \theta_n} = \ii^n$. It is therefore possible for any given $g$ and $P$ to have a canonical choice of $\Phi$, which satisfies $b_n > 0$, by taking
\begin{displaymath}
    \varphi_n(x) = \frac{\ii^n}{\sqrt{2\pi}} \int_{-\infty}^\infty \ee^{\ii x \xi} p_n(\xi) g(\xi)  \D \xi.
  \end{displaymath}

We can also produce a unique canonical orthonormal system from an absolutely continuous measure $\mathrm{d}\mu(\xi) = w(\xi) \mathrm{d}\xi$ on the real line, where $w(\xi)$ decays superalgebraically fast as $|\xi| \to \infty$. Specifically, the functions
\begin{equation}
    \label{p_to_phi}
    \varphi_n(x) = \frac{\ii^n}{\sqrt{2\pi}} \int_{-\infty}^\infty \ee^{\ii x \xi} p_n(\xi) |w(\xi)|^{\frac12} \D \xi
\end{equation}
form an orthonormal system in $\CC{L}_2(\BB{R})$ with a tridiagonal, irreducible skew-Hermitian differentiation matrix with a positive superdiagonal. The system is dense in $\CC{L}_2(\BB{R})$ if $P$ is dense in $L^2(\BB{R},w(\xi)\mathrm{d}\xi)$.

\subsection{Computing $\Phi$}

We proved in \cite{iserles18osssd} that any system $\Phi$ of $\CC{L}_2(\BB{R})\cap\CC{C}^\infty(\BB{R})$ functions that obey \R{eq:skew} obeys the relation
\begin{equation}
  \label{SkewDervs}
  \varphi_n(x)=\frac{1}{b_0b_1\cdots b_{n-1}} \sum_{\ell=0}^{\lfloor n/2\rfloor} \alpha_{m,\ell} \varphi_0^{(n-2\ell)}(x),\qquad n\in\BB{Z}_+,
\end{equation}
where
\begin{displaymath}
  \alpha_{n+1,0}=1,\qquad \alpha_{n+1,\ell}=b_{n-1}^2\alpha_{n-1,\ell-1}+\alpha_{n,\ell},\quad \ell=1\ldots\left\lfloor\frac{n+1}{2}\right\rfloor\!.
\end{displaymath}
Our setting lends itself to similar representation, which follows from \R{eq:sHerm} by induction.

\begin{lemma}
  The functions $\Phi$ consistent with \R{eq:sHerm} satisfy the relation
  \begin{equation}
    \label{sHermDervs}
    \varphi_n(x)=\frac{1}{b_0b_1\cdots b_{n-1}} \sum_{\ell=0}^n \beta_{n,\ell} \varphi_0^{(\ell)}(x),\qquad n\in\BB{Z}_+,
  \end{equation}
  where $\beta_{0,0}=\beta_{1,1}=1$ , $\beta_{1,0}=-\ii c_0$ and
  \begin{eqnarray*}
    \beta_{n+1,0}&=&b_{n-1}^2\beta_{n-1,0}-\ii c_n\beta_{n,0},\\
    \beta_{n+1,\ell}&=&\beta_{n,\ell-1}+b_{n-1}^2\beta_{n-1,\ell}-\ii c_n\beta_{n,\ell},\quad \ell=1,\ldots,n+1
  \end{eqnarray*}
  for $n\in\BB{N}$.
\end{lemma}

Like \R{SkewDervs}, the formula \R{sHermDervs} is often helpful in the calculation of $\varphi_1,\varphi_2,\ldots$ once $\varphi_0$ is known. The obvious idea is to compute explicitly the derivatives of $\varphi_0$ and form their linear combination \R{sHermDervs}, but equally useful is a generalisation of an approach originating in \cite{iserles18osssd}. Thus, Fourier-transforming \R{sHermDervs},
\begin{displaymath}
  \hat{\varphi}_n(\xi)=\frac{\hat{\varphi}_0(\xi)}{b_0b_1\cdots b_{n-1}} \sum_{\ell=0}^n \beta_{n,\ell}(\ii\xi)^\ell.
\end{displaymath}
On the other hand, Fourier transforming \R{p_to_phi}, we have
\begin{displaymath}
  \hat{\varphi}_n(\xi)=\ii^n |w(\xi)|^{1/2} p_n(\xi).
\end{displaymath}
Our first conclusion is that $\hat{\varphi}_0(\xi)= |w(\xi)|^{1/2}/p_0$. Moreover, comparing the two displayed equations,
\begin{equation}
  \label{ExpRep0}
  \frac{1}{b_0b_1\cdots b_{n-1}}\sum_{\ell=0}^n \beta_{n,\ell} (\ii\xi)^\ell =\frac{\ii^n}{p_0} p_n(\xi).
\end{equation}
The polynomials $p_n$ are often known explicitly. In that case it is helpful to rewrite \R{sHermDervs}  in a more explicit form.

\begin{lemma}
  \label{ExplicitPhi}
  Suppose that   $p_n(\xi)=\sum_{\ell=0}^n p_{n,\ell} \xi^\ell$, $n\in\BB{Z}_+$. Then
  \begin{equation}
    \label{ExpRep}
    \varphi_n(x)=\frac{\ii^n}{p_{0,0}}\sum_{\ell=0}^n (-\ii)^\ell p_{n,\ell} \varphi_0^{(\ell)}(x),\qquad n\in\BB{Z}_+.
  \end{equation}
\end{lemma}

\begin{proof}
  By \R{sHermDervs},  substituting  the explicit form of $p_n$ in \R{ExpRep0}.
\end{proof}

\subsection{An example}

The next section is concerned with the substantive example of a system $\Phi$ with a skew-Hermitian differentiation matrix that originates in the Fourier setting once we use a Laguerre measure. What, though, about other examples? Once we turn our head to generating explicit examples of orthonormal systems in the spirit of this paper and of \cite{iserles18osssd}, we are faced with a problem: all steps in subsections 2.1--3 must be generated explicitly. Thus, we must choose an absolutely continuous measure for which the recurrence coefficients in \R{TTRR} are known explicitly, compute explicitly $\{p_n\}_{n\in\BB{Z}_+}$ and either 
\begin{itemize}
\item compute explicitly $\varphi_0(x)=(2\pi)^{-1/2}p_0 \int_{-\infty}^\infty |w(\xi)|^{1/2}\ee^{\ii x\xi}\D\xi$ and its derivatives, subsequently forming \R{ExpRep} and manipulating it further into a user-friendly form,
\end{itemize}
or
\begin{itemize}
\item compute explicitly \R{p_to_phi} for all $n\in\BB{Z}_+$.
\end{itemize}
Either course of action is restricted by the limitations on our knowledge of {\em explicit\/} fomul\ae{} of orthogonal polynomials for absolutely continuous measures (thereby excluding, for example, Charlier and Lommel polynomials, as well as the Askey--Wilson hierarchy). Thus Hermite polynomials and their generalisations \cite{iserles18osssd}, Jacobi and Konoplev polynomials  \cite{iserles18osssd}, Carlitz polynomials \cite{iserles19fao} and, in the next section, Laguerre polynomials. 

Herewith we present another example which, albeit of no apparent practical use, by its very simplicity helps to illustrate our narrative. Let $\alpha\in\BB{R}$ and consider $\D\mu(\xi)=\ee^{-(\xi-\alpha)^2}\D\xi$, a shifted Hermite measure. The underlying orthonormal set consists of
\begin{displaymath}
  p_n(x)=\frac{1}{\sqrt{2^nn!\sqrt{\pi}}} \CC{H}_n(x-\alpha),\qquad n\in\BB{Z}_+,
\end{displaymath}
therefore
\begin{displaymath}
  \xi p_n(\xi)=\sqrt{\frac{n}{2}}p_{n-1}(\xi)+\alpha p_n(\xi)+\sqrt{\frac{n+1}{2}} p_{n+1}(\xi)
\end{displaymath}
-- we deduce that $b_n=\sqrt{(n+1)/2}$ and $c_n\equiv\alpha$ in \R{eq:sHerm}. Moreover,
\begin{displaymath}
  \varphi_n(x)=\frac{\ii^n}{\sqrt{2\pi}} \frac{1}{\sqrt{2^nn!\sqrt{\pi}}} \int_{-\infty}^\infty \CC{H}_n(\xi-\alpha)\ee^{-(\xi-\alpha)^2/2-\ii x\xi}\D\xi =\frac{\ee^{-x^2/2-\ii\alpha x}}{\sqrt{2^nn!\sqrt{\pi}}} \CC{H}_n(x),
\end{displaymath}
`twisted' Hermite functions. It is trivial to confirm that they satisfy \R{eq:sHerm} or  derive them directly from \R{ExpRep}.

\subsection{Connections to chromatic expansions}

Theorem 1 characterises all systems $\Phi$ in $\CC{L}_2(\BB{R})$ satisfying equation \eqref{eq:sHerm}. These systems depend on a family of orthonormal polynomials on the real line with associated measure $\D\mu$ and a function $g$ on the real line. Theorems 2 and 3 focussed on the special case in which $\D\mu(\xi) = |g(\xi)|^2 \D\xi$. This special case yields all systems which are orthonormal in the inner standard product, which turn out to be complete in a Paley-Wiener space.

An anonymous referee has made the authors aware of a considerable amount of work devoted to the special case in which $\D\mu(\xi) = g(\xi) \D\xi$, whose systems generate so-called \emph{chromatic expansions} \cite{ignjatovic2007local,zayed2009generalizations,zayed2011chromatic,zayed2014chromatic}. These systems have some remarkable properties for application to signal processing which we summarise here whilst making connections to the present work.

Given a sequence of orthonormal polynomials $\{P_n\}_{n\in\BB{Z}_+}$ with respect to a finite Borel measure $\D\mu$ on the real line, define the function
\begin{equation}
\psi(x) := \int_{-\infty}^\infty \ee^{\ii x \xi} \,\D\mu(\xi),
\end{equation}
and the operators
\begin{equation}
\mathcal{K}_n = P_n\left( \ii \frac{\D}{\D x}\right),
\end{equation}
for all $n\in\BB{Z}_+$ acting on $C^\infty(\BB{R})$. The chromatic expansion of a smooth function $f$ is the formal series
\begin{equation}
f(x) = \sum_{n=0}^\infty \mathcal{K}_n[f](0) \mathcal{K}_n[\psi](x),
\end{equation}
which converges uniformly on the real line if, for example, $\mu$ is such that $\psi$ is analytic in a strip around the real axis, $f$ is analytic in this strip too, and the sequence of coefficients $\{\mathcal{K}_n[f](0)\}_{n\in\BB{Z}_+}$ is in $\ell_2$ \cite{zayed2014chromatic}.

The connection to the present work is as follows. For all $n\in\BB{Z}_+$, let $\varphi_n = \mathcal{K}_n[\psi]$ be the elements of the system $\Phi$. Then $\Phi$ is of the form described in Theorem 1 with $g(\xi) \D\xi = \D\mu(\xi)$.

This advantages for signal processing are twofold:
\begin{itemize}
  \item The expansion coefficients are given explicitly and depend locally on the function $f$ centred around the point $0$. The expansions can be made local to points other than $0$ as in \cite{ignjatovic2007local}.
  \item The functions $\Phi$ are bandlimited, with Fourier transforms supported precisely on the support of $\mu$.
  \end{itemize}

While the basis $\Phi$ with $\varphi_n = \mathcal{K}_n[\psi]$ is not orthonormal in the standard inner product on $\CC{L}_2(\BB{R})$, under some mild assumptions on $\mu$, it is possible to show that $\Phi$ is orthonormal with respect to the inner product
\begin{equation}
\langle f,g\rangle_{\mathcal{K}} = \sum_{n=0}^\infty \mathcal{K}_n[f](0)\overline{\mathcal{K}_n[g](0)},
\end{equation}
and is complete in a space of analytic functions on the real line for which the induced norm is finite \cite{ignjatovic2007local}.

\setcounter{equation}{0}
\section{The Fourier--Laguerre basis}

\subsection{A general expression}

A skew-Hermite setting allows an important generalisation of the narrative of \cite{iserles18osssd}, namely to Borel measures in the Fourier space which are not symmetric. The most obvious instance is the {\em Laguerre measure\/} $\D\mu(\xi)=\chi_{(0,\infty)}(\xi) \xi^\alpha \ee^{-\xi}\D \xi$, where $\alpha>-1$. The corresponding orthogonal polynomials are the {\em (generalised) Laguerre polynomials\/}
\begin{equation}
  \label{genLaguerre}
  \LL_n^{(\alpha)}(\xi)=\frac{(1+\alpha)_n}{n!} \hyper{1}{1}{-n}{1+\alpha}{\xi}=\frac{(1+\alpha)_n}{n!} \sum_{\ell=0}^n (-1)^\ell {n\choose \ell} \frac{\xi^\ell}{(1+\alpha)_\ell},
\end{equation}
where $(z)_m=z(z+1)\cdots(z+m-1)$ is the {\em Pochhammer symbol\/} and ${}_1F_1$ is a {\em confluent hypergeometric function\/} \cite[p.~200]{rainville60sf}. The Laguerre polynomials obey the recurrence relation
\begin{displaymath}
  (n+1)\LL_{n+1}^{(\alpha)}(\xi)=(2n+1+\alpha-\xi)\LL_n^{(\alpha)}(\xi)-(n+\alpha)\LL_{n-1}^{(\alpha)}(\xi).
\end{displaymath}
First, however, we need to recast them in a form suitable to the analysis of Section~2 -- specifically, we need to renormalise them so that they are orthonormal and so that the coefficient of $\xi^n$ in $p_n$ is positive. Since
\begin{displaymath}
  \|\LL_n^{(\alpha)}\|^2=\int_0^\infty \xi^\alpha [\LL_n^{(\alpha)}(\xi)]^2\ee^{-\xi}\D\xi=\frac{\Gamma(n+1+\alpha)}{n!}
\end{displaymath}
\cite[p.~206]{rainville60sf} and the sign of $\xi^n$ in \R{genLaguerre} is $(-1)^n$, we set
\begin{displaymath}
  p_n(\xi)=(-1)^n \sqrt{\frac{n!}{\Gamma(n+1+\alpha)}} \LL_n^{(\alpha)}(\xi),\qquad n\in\BB{Z}_+. 
\end{displaymath}
We deduce after simple algebra that
\begin{displaymath}
  b_n=\beta_n=\sqrt{(n+1)(n+1+\alpha)},\qquad c_n=\delta_n=2n+1+\alpha
\end{displaymath}
in \R{eq:sHerm} and \R{TTRR}. ($b_n=\beta_n$ because the latter is real and positive.) 

To compute $\Phi$ we note that, letting $\tau=(\frac12-\ii x)\xi$, \R{p_to_phi} yields
\begin{eqnarray*}
  \varphi_0(x)&=&\frac{1}{\sqrt{2\pi\Gamma(1+\alpha)}} \int_0^\infty \xi^{\alpha/2}\ee^{-\xi/2+\ii\xi x}\D \xi\\
  &=&\frac{1}{\sqrt{2\pi(1+\alpha)}} \frac{2^{\alpha/2+1}}{(1-2\ii x)^{\alpha/2+1}} \int_0^\infty\! \tau^{\alpha/2}\ee^{-\tau}\D\tau \\
  &=&\frac{1}{\sqrt{2\pi}} \frac{\Gamma(1+\frac12\alpha)}{\sqrt{\Gamma(1+\alpha)}}\left(\frac{2}{1-2\ii x}\right)^{\!1+\alpha/2}.
\end{eqnarray*}
It now follows by simple induction that\footnote{Note that the bracketed superscripts $^{(\alpha)}$ and $^{(\ell)}$ have different meanings. The former is the standard notation for the parameter in the generalized Laguerre polynomial and the latter is the standard notation for the $\ell$th derivative.}
\begin{displaymath}
  \varphi_0^{(\ell)}(x)=\frac{\ii^\ell}{\sqrt{2\pi}} \frac{\Gamma(\ell+1+\frac12\alpha)}{\sqrt{\Gamma(1+\alpha)}}\left(\frac{2}{1-2\ii x}\right)^{\!\ell+1+\alpha/2},\qquad \ell\in\BB{Z}_+. 
\end{displaymath}
Moreover,
\begin{eqnarray*}
  p_n(\xi)&=&(-1)^n \sqrt{\frac{n!}{\Gamma(n+1+\alpha)}} (1+\alpha)_n \sum_{\ell=0}^n \frac{(-1)^\ell \xi^\ell}{\ell!(n-\ell)!(1+\alpha)_\ell}\\
  &=&\frac{\sqrt{n!\Gamma(n+1+\alpha)}}{\Gamma(1+\alpha)} \sum_{\ell=0}^n \frac{(-1)^{n-\ell}\xi^\ell}{\ell!(n-\ell)!(1+\alpha)_\ell},
\end{eqnarray*}
therefore
\begin{displaymath}
  p_{n,\ell}=\frac{\sqrt{n!\Gamma(n+1+\alpha)}}{\Gamma(1+\alpha)}  \frac{(-1)^{n-\ell}}{\ell!(n-\ell)!(1+\alpha)_\ell},\qquad \ell=0,\ldots,n
\end{displaymath}
and substitution in \R{ExpRep} gives
\begin{eqnarray*}
  \varphi_n(x)&=&\frac{(-\ii)^n}{\sqrt{2\pi}} \frac{\sqrt{n!\Gamma(n+1+\alpha)}}{\Gamma(1+\alpha)}\! \left(\frac{2}{1-2\ii x}\right)^{\!1+\frac{\alpha}{2}} \!\sum_{\ell=0}^n  \frac{\Gamma(\ell+1+\frac12\alpha)}{\ell!(n-\ell)!(1+\alpha)_\ell} \!\left(\frac{2}{1-2\ii x}\right)^{\!\ell}\\
  &=&\frac{(-\ii)^n}{\sqrt{2\pi}} \sqrt{\frac{\Gamma(n+1+\alpha)}{n!}} \frac{\Gamma(1+\frac{\alpha}{2})}{\Gamma(1+\alpha)} \left(\frac{2}{1-2\ii x}\right)^{\!1+\frac{\alpha}{2}} \hyper{2}{1}{-n,1+\frac12\alpha}{1+\alpha}{\frac{2}{1-2\ii x}}\!.
\end{eqnarray*}
The identity,
\begin{displaymath}
  \hyper{2}{1}{-n,b}{c}{z}=\frac{(c-b)_n}{(c)_n}\hyper{2}{1}{-n,b}{b-c-n+1}{1-z},
\end{displaymath}
\cite[15.8.7]{DLMF}, implies that we have
\begin{displaymath}
  \varphi_n(x)= \frac{(-\ii)^n}{\sqrt{2\pi}}  \frac{\alpha 2^{\frac{\alpha}{2}}\Gamma(n+\frac{\alpha}{2})}{\sqrt{n!\Gamma(n+1+\alpha)}}\left(\frac{1}{1-2\ii x}\right)^{\!1+\frac{\alpha}{2}} \hyper{2}{1}{-n,1+\frac12\alpha}{1-\frac12\alpha-n}{-\frac{1+2\ii x}{1-2\ii x}}\!.
\end{displaymath}
It is now clear that $\varphi_n$ is proportional to $(1-2\ii x)^{-1-\alpha/2}$ times a polynomial of degree $n$ in the expression $(1+2\ii x)/(1-2\ii x)$ i.e.
\begin{equation}
\label{FL}
\varphi_n(x) = (-\ii)^n\sqrt{\frac{2}{\pi}} \left(\frac{1}{1-2\ii x}\right)^{\!1+\frac{\alpha}{2}} \Pi^{(\alpha)}_n\left(\frac{1+2\ii x}{1-2\ii x}\right), 
\end{equation}
where $\Pi^{(\alpha)}_n$ is a polynomial of degree $n$. Using the substitution $x = \frac12 \tan \frac{\theta}{2}$ for $\theta \in (-\pi,\pi)$, which implies $(1+2\ii x)/(1-2\ii x) = \ee^{\ii \theta}$, the orthonormality of the basis $\Phi$ can be seen to imply that $\{\Pi^{(\alpha)}_n\}_{n\in\BB{Z}_+}$ are in fact orthogonal polynomials on the unit circle (OPUC) with respect to the weight
\begin{displaymath}
  W(\theta) = \cos^{\alpha} \frac{\theta}{2}.
\end{displaymath}
To be clear, this means that for all $n,m\in\BB{Z}_+$,
\begin{displaymath}
  \frac{1}{2\pi}\int_{-\pi}^\pi \overline{\Pi^{(\alpha)}_n(\ee^{\ii \theta})}\Pi^{(\alpha)}_m(\ee^{\ii \theta}) \,\cos^{\alpha} \frac{\theta}{2}\,\mathrm{d} \theta = \delta_{n,m}.
\end{displaymath}
These polynomials are related to the \emph{Szeg\H{o}--Askey polynomials} \cite[18.33.13]{DLMF}, $\{\phi^{(\lambda)}_n\}_{n\in\BB{Z}_+}$, which satisfy
\begin{displaymath}
  \frac{1}{2\pi}\int_{-\pi}^\pi \overline{\phi^{(\lambda)}_n(\ee^{\ii \theta})}\phi^{(\lambda)}_m(\ee^{\ii \theta}) \,(1-\cos \theta)^\lambda\,\mathrm{d} \theta = \delta_{n,m},
\end{displaymath}
by the relation $\Pi^{(\alpha)}_n(z) \propto \phi^{(\frac{\alpha}{2})}_n(-z)$. The Szeg\H{o}--Askey polynomials are known to satisfy a Delsarte--Genin relationship to the Jacobi polynomials $\CC{P}_n^{\left(\frac{\alpha-1}{2},-\frac12 \right)}$ and $\CC{P}_n^{\left(\frac{\alpha+1}{2},\frac12 \right)}$ due to the symmetry of the weight of orthogonality \cite[p.~295]{szego75op}, \cite[18.33.14]{DLMF}. Specifically,
\begin{eqnarray*}
  \ee^{-n\ii\theta}\Pi^{(\alpha)}_{2n}(\ee^{\ii\theta}) &=& A_n \CC{P}_n^{\left(\frac{\alpha-1}{2},-\frac12 \right)}\left(\cos \theta\right) + \ii B_n \sin \theta  \CC{P}_{n-1}^{\left(\frac{\alpha+1}{2},\frac12 \right)}\left(\cos \theta\right), \\
  \ee^{(1-n)\ii\theta} \Pi^{(\alpha)}_{2n-1}(\ee^{\ii \theta}) &=& C_n \CC{P}_n^{\left(\frac{\alpha-1}{2},-\frac12 \right)}\left(\cos \theta\right) + \ii D_n\sin \theta \CC{P}_{n-1}^{\left(\frac{\alpha+1}{2},\frac12 \right)}\left(\cos\theta\right),
\end{eqnarray*}
for some real constants $\{A_n,B_n,C_n,D_n\}_{n\in\BB{Z}_+}$. It is therefore possible to express the functions $\Phi$ in terms of Jacobi polynomials; this is something we will not pursue here, but could be of interest for further research. In what follows we will restrict ourselves to the case $\alpha = 0$, which is extremely simple.

We are not aware if this connection between the general Laguerre polynomials and  Szeg\H{o}--Askey polynomials (and hence  Jacobi polynomials) via the Fourier transform has been acknowledged before in the literature.

\subsection{The Malmquist--Takenaka system}

The expression \R{FL} comes into its own once we let $\alpha=0$, namely consider the `simple' Laguerre polynomials $\LL_n$. Now $W(\theta) \equiv 1$ and so $\Pi^{(\alpha)}_n(z) = z^n$. We have $b_n=n+1$, $c_n=2n+1$ and
\begin{displaymath}
  \varphi_n(x)=\sqrt{\frac{2}{\pi}} \ii^n \frac{(1+2\ii x)^n}{(1-2\ii x)^{n+1}},\qquad n\in\BB{Z}_+.
\end{displaymath}
The factor of $(-1)^n$ which might appear to have been added here comes from the identity $\hyper{2}{1}{-n,1}{1}{z} = (1-z)^n$ with $z = 2/(1-2\ii x)$. Alternatively, we may apply a formula for the Laplace transform of Laguerre polynomials at an appropriate point in the complex plane \cite[18.17.34]{DLMF}, to obtain
\begin{displaymath}
  \int_0^\infty \LL_n(\xi) \ee^{-\frac{\xi}{2}+\ii x\xi}\D\xi =2(-1)^n\frac{(1+2\ii x)^n}{(1-2\ii x)^{n+1}},\qquad n\in\BB{Z}_+. 
\end{displaymath}

By Theorem \ref{thm:PW}, these functions are dense in the Paley--Wiener space $\mathcal{PW}_{[0,\infty)}(\BB{R})$. To obtain a basis for the whole of $\mathrm{L}_2(\BB{R})$, we must add to this a basis for $\mathcal{PW}_{(-\infty,0]}(\BB{R})$. The obvious way to do so is to consider the same functions as above, but for the orthogonal polynomials with respect to $\chi_{(-\infty,0]}(\xi) \ee^\xi\D\xi$, which are precisely $\mathrm{L}_n(-\xi)$, $n\in\BB{Z}_+$. Using the Laplace transform again, this leads to the functions
\begin{eqnarray*}
  \tilde{\varphi}_n(x) &=& \frac{(-\ii)^n}{\sqrt{2\pi}}\int_{-\infty}^0 \ee^{\ii x\xi}\,\mathrm{L}_n(-\xi) \,\ee^{\frac{\xi}{2}} \, \D \xi =\frac{(-\ii)^n}{\sqrt{2\pi}}\int_0^\infty \ee^{-\ii x\xi}\,\mathrm{L}_n(\xi) \,\ee^{-\frac{\xi}{2}} \, \D \xi \\
  &=& \ii^n \sqrt{\frac{2}{\pi}} \frac{(1-2\ii x)^{\!n}}{(1+2\ii x)^{ n+1}}.
\end{eqnarray*}
Letting $\varphi_n=\tilde{\varphi}_{-n-1}$, $n\leq-1$, we obtain the {\em Malmquist--Takenaka system\/} \R{MT}. 

As a matter of historical record, \citeasnoun{malmquist26stc} and \citeasnoun{takenaka25oof} considered a more general system of the form
\begin{displaymath}
  \mathcal{B}_n(z)=\frac{\sqrt{1-|\theta_n|^2}}{1-\overline{\theta}_n z} \psi_n(z),\quad \mathcal{B}_{-n}(z)=\overline{\mathcal{B}_n(1/\overline{z})},\qquad n\in\BB{Z}_+,
\end{displaymath}
where $\psi_n(z)=\prod_{k=0}^{n-1}(z-\theta_k)/(1-\overline{\theta}_k z)$ is a finite Blaschke product and $|\theta_k|<1$, $k\in\BB{Z}_+$. The nature of the questions they have asked was different -- essentially, they proved that the above system is a basis (which need not be orthogonal) of $\mathcal{H}_2$, the Hardy space of complex analytic functions in the open unit disc. In our case the $\theta_k\equiv2\ii$ are all the same and outside the unit circle, yet it seems fair (and consistent with, say, \cite{pap15ecm}) to call \R{MT} a Malmquist--Takenaka system. 

\begin{figure}[! htb]
  \begin{center}
    \begin{minipage}{\textwidth} 
      \begin{small}
      \begin{overpic}[width=\textwidth]{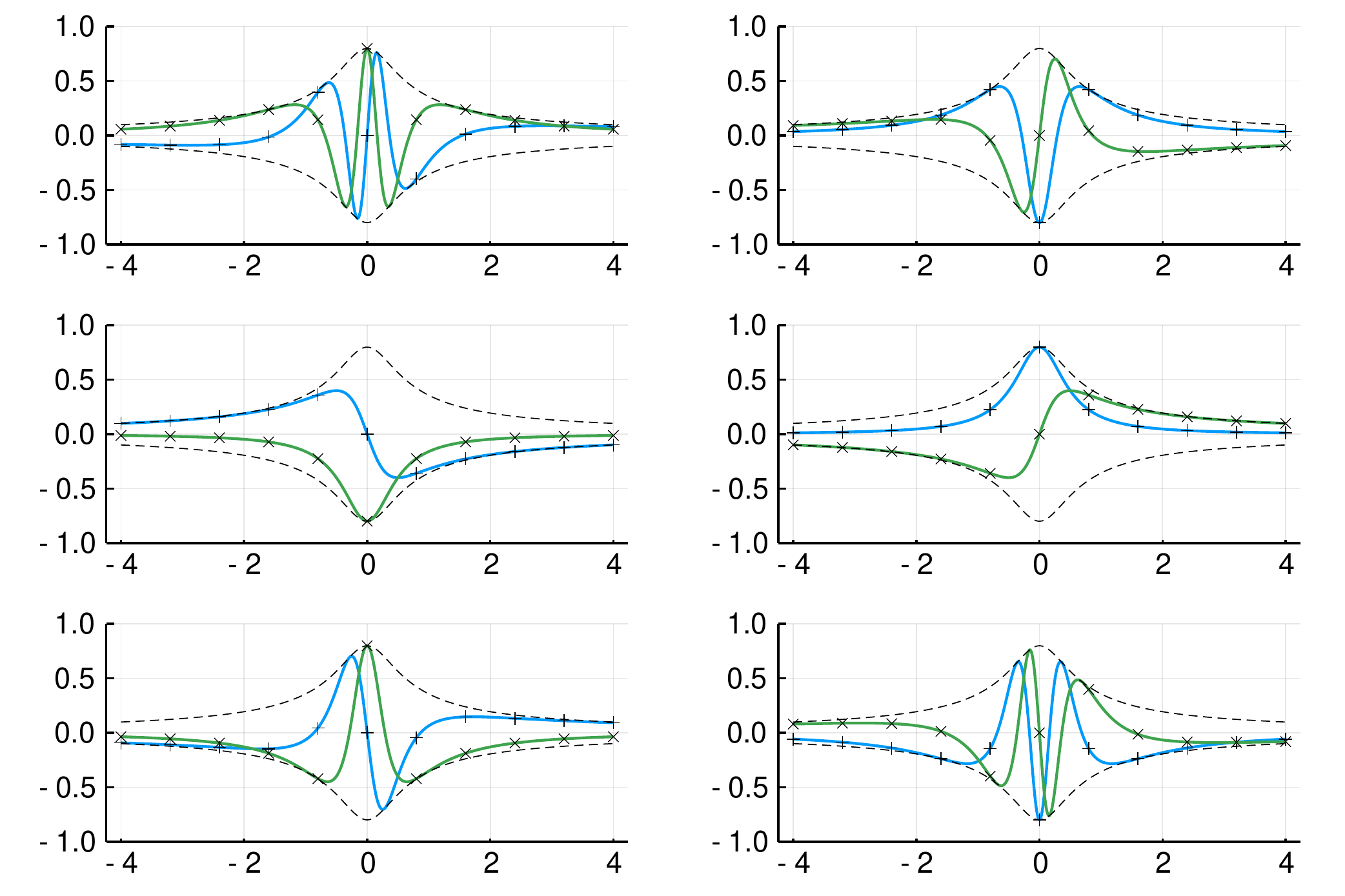}
        \put(24,66) {$n=-3$}
        \put(74,66) {$n=-2$}
        \put(24,43) {$n=-1$} 
        \put(75,43) {$n=0$} 
        \put(25,21) {$n=1$} 
        \put(75,21) {$n=2$} 
      \end{overpic} 
    \end{small}
    \end{minipage} 
    \caption{Real and imaginary parts (blue lines with black `$+$'s and green lines with black `$\times$'s, respectively) of the MT functions $\varphi_n$ for $n=-3,\ldots,2$. The envelope of $\pm\sqrt{\frac{2}{\pi(1+4x^2)}}$ is also plotted as a dashed line.}
    \label{fig:3.1}
  \end{center}
\end{figure}

Fig.~\ref{fig:3.1} displays the real and imaginary parts of few Malmquist--Takenaka functions.

Let us dwell briefly on the properties of \R{MT}. 
\begin{itemize}
\item The system is dense in $\CC{L}_2(\BB{R})$, because standard Laguerre polynomials are dense in $\CC{L}_2((0,\infty),\ee^{-\xi}\D\xi)$ and $\{\LL_n(-\xi)\}_{n\in\BB{Z}_+}$ is dense in $\CC{L}_2((-\infty,0),\ee^{\xi}\D\xi)$. 
\item All the functions $\varphi_n$ are uniformly bounded,
\begin{displaymath}
  |\varphi_n(x)|= \sqrt{\frac{2}{\pi}} \frac{1}{\sqrt{1+4x^2}},\qquad x\in\BB{R}.
\end{displaymath}
\item The differentiation matrix,
\begin{equation}
  \label{matrixD}
  \mathcal{D}=
  \left[
  \begin{array}{cccccccc}
        \ddots & \ddots\\
        \ddots & -5\ii & -2\\
        & 2 & -3\ii & -1\\
        & & 1 & -\ii & 0\\
        & & & 0 & \ii & 1\\
        & & & & -1 & 3\ii & 2\\
        & & & & & -2 & 5\ii & \ddots\\
        & & & & & & \ddots & \ddots 
  \end{array}
  \right]\!,
\end{equation}
is skew-Hermitian, tridiagonal and {\em reducible\/} -- specifically, $\mathcal{D}_{-1,1}=\mathcal{D}_{1,-1}=0$ and the matrix decomposes into two irreducible `chunks', corresponding to $n\leq-1$ and $n\geq0$. 

While \R{matrixD} follows from our construction, it can be also proved directly from \R{MT}:
\begin{eqnarray*}
  \varphi_n'(x)&=&\ii^n \sqrt{\frac{2}{\pi}} \left[ 2\ii n \frac{(1+2\ii x)^{n-1}}{(1-2\ii x)^{n+1}} +2\ii(n+1) \frac{(1+2\ii x)^n}{(1-2\ii x)^{n+2}}\right]\\
  &=&\ii^{n+1} \sqrt{\frac{2}{\pi}} \frac{(1+2\ii x)^{n-1}}{(1-2\ii x)^{n+2}} [2n(1-2\ii x)+2(n+1)(1+2\ii x)]\\
  &=&\ii^{n+1} \sqrt{\frac{2}{\pi}} \frac{(1+2\ii x)^{n-1}}{(1-2\ii x)^{n+2}} (4n+2+4\ii x),
\end{eqnarray*}
while
\begin{eqnarray*}
  &&-n\varphi_{n-1}(x)+(2n+1)\ii \varphi_n(x)+(n+1)\varphi_{n+1}(x)\\
  &=&\ii^{n+1} \sqrt{\frac{2}{\pi}} \frac{(1+2\ii x)^{n-1}}{(1-2\ii x)^{n+2}} [n(1-2\ii x)^2+(2n+1) (1+4x^2)+(n+1)(1+2\ii x)^2]\\
  &=&\ii^{n+1} \sqrt{\frac{2}{\pi}} \frac{(1+2\ii x)^{n-1}}{(1-2\ii x)^{n+2}} (4n+2+4\ii x).
\end{eqnarray*}
\item The MT system obeys a host of identities that make it amenable for implementation in spectral methods. The following were identified by Christov,
\begin{eqnarray}
  \label{product}
  \varphi_m(x)\varphi_n(x)&=&\frac{1}{\sqrt{2\pi}}[\varphi_{m+n}(x)-\ii \varphi_{m+n+1}(x)],\qquad m,n\in\BB{Z}_+,\\
  \nonumber
  x\varphi_n'(x)&=&-\frac{n}{2}\ii \varphi_{n-1}(x)-\frac12\varphi_n(x)-\frac{n+1}{2}\varphi_{n+1}(x),\qquad n\in\BB{Z}
\end{eqnarray}
\cite{christov1982complete} and the following is apparently new,
\begin{displaymath}
  \frac{4\ii}{1+4x^2} \varphi_n(x)=-\varphi_{n-1}(x)+2\varphi_n(x)+\varphi_{n+1}(x),\qquad n\in\BB{Z}.
\end{displaymath}
In particular, \R{product} implies that
\begin{eqnarray*}
  &&\sum_{m=-\infty}^\infty \hat{f}_m\varphi_m(x) \sum_{n=-\infty}^\infty \hat{h}_n\varphi_n(x) \\
  &=&\frac{1}{\sqrt{2\pi}} \sum_{n=-\infty}^\infty \left[\sum_{m=-\infty}^\infty \hat{f}_m(\hat{h}_{n-m}-\ii \hat{h}_{n-m-1})\right]\! \varphi_n(x),
\end{eqnarray*}
allowing for an easy multiplication of expansions in the MT basis.
\end{itemize}

\subsection{Expansion coefficients}

Arguably the most remarkable feature of the MT system is that expansion coefficients can be computed very rapidly indeed. Thus, let $f\in\CC{L}_2(\BB{R})$. Then
\begin{displaymath}
  f(x)=\sum_{n=-\infty}^\infty \hat{f}_n \varphi_n(x)\qquad \mbox{where}\qquad \hat{f}_n=\int_{-\infty}^\infty f(x) \overline{\varphi_n(x)}\D x,\quad n\in\BB{Z}.
\end{displaymath}
We do not dwell here on speed of convergence except for brief comments in subsection~3.4 -- this is a nontrivial issue and, while general answer is not available, there is wealth of relevant material in \cite{weideman95tao}. Our concern is with efficient algorithms for the evaluation of $\hat{f}_n$ for $-N\leq n\leq N-1$. 

The key observation is that 
\begin{displaymath}
  \varphi_n(x)=\ii^n \sqrt{\frac{2}{\pi}} \frac{1}{1-2\ii x} \left(\frac{1+2\ii x}{1-2\ii x}\right)^{\!n}
\end{displaymath}
and the term on the right is of unit modulus. We thus change variables
\begin{equation}
  \label{eqn:unitcircle}
  \frac{1+2\ii x}{1-2\ii x}=\ee^{\ii\theta},\qquad -\pi<\theta<\pi,
\end{equation}
in other words $x=\frac12 \tan\frac{\theta}{2}$ and, in the new variable
\begin{displaymath}
  \varphi_n(x)=\ii^n \sqrt{\frac{2}{\pi}} \ee^{\ii(n+\frac12)\theta} \cos\frac{\theta}{2},\qquad n\in\BB{Z}. 
\end{displaymath}
We deduce that
\begin{equation}
  \label{expand}
  \hat{f}_n =\frac{(-\ii)^n}{2\sqrt{2\pi}} \int_{-\pi}^\pi \left(1-\ii\tan\frac{\theta}{2}\right) f\!\left(\frac12 \tan\frac{\theta}{2}\right) \ee^{-\ii n\theta}\D\theta,\qquad n\in\BB{Z},
\end{equation}
a Fourier integral. Two immediate consequences follow. Firstly, the convergence of the coefficients as $|n|\rightarrow\infty$ is dictated by the smoothness of the modified function
\begin{displaymath}
  F(\theta) = \left(1-\ii \tan\frac{\theta}{2} \right)f\left(\frac12 \tan \frac{\theta}{2} \right)\!,\qquad -\pi<\theta<\pi.
\end{displaymath}
Secondly, provided $F$ is analytic, \R{expand} can be approximated to exponential accuracy by a {\em Discrete Fourier Transform\/}\footnote{The approximation remains valid -- but less accurate -- for $F\in\CC{C}^k(-\pi,\pi)$.} and this, in turn, can be computed rapidly with {\em Fast Fourier Transform (FFT)}: the first $N$ coefficients require $\O{N\log_2N}$ operations. 

\begin{proposition}[Fast approximation of coefficients]
  The truncated MT system $\{\varphi_n\}_{n=-N}^{N-1}$ is orthonormal with respect to the  discrete inner product,
  \begin{displaymath}
    \langle f, g \rangle_{N} = \frac{\pi}{4N} \sum_{k=1}^{2N} f\left(x_k\right) \overline{g\left(x_k\right)}(1+4x_k^2), 
  \end{displaymath}
  where
  \begin{displaymath}
    x_k = \frac{1}{2}\tan\frac{\theta_k}{2}, \qquad k = 1,2,\ldots,2N,
  \end{displaymath} 
  and $\theta_1,\ldots,\theta_{2N}$ are equispaced points in the periodic interval $[-\pi,\pi]$ (such that $\theta_k - \theta_{k-1} = \pi/N$). The coefficients of a function $f$ in the span of $\{\varphi_n\}_{n=-N}^{N-1}$ are exactly equal to
  \begin{equation}
  	\label{eqn:discretecoeffs}
    \langle f, \varphi_n \rangle = \langle f, \varphi_n \rangle_N = (-\ii)^n \sqrt{\frac{\pi}{2}} \frac{1}{2N}\sum_{k=1}^{2N} f(x_k)(1-2\ii x_k) \ee^{-\ii n \theta_k},
  \end{equation}
  and can be computed simultaneously in $\O{N\log_2N}$ operations using the FFT.
\end{proposition}

\begin{proof}
  Let $k,\ell$ be integers satisfying $-N \leq k,\ell \leq N-1$. Then
  \begin{displaymath}
    \langle \varphi_\ell, \varphi_k \rangle_N = \frac{1}{2N} \sum_{j=1}^{2N} \left(\frac{1+2\ii x_j}{1-2\ii x_j} \right)^{\!\ell-k}.
  \end{displaymath}
  If $k= \ell$ then this is clearly equal to 1. Otherwise, using equation \eqref{eqn:unitcircle}, we see that,
  \begin{displaymath}
    \langle \varphi_\ell, \varphi_k \rangle_N =  \frac{1}{2N} \sum_{j=1}^{2N} \ee^{\ii(\ell-k)\theta_j}.
  \end{displaymath}
  Summing the geometric series, since $\theta_j - \theta_{j-1} = \pi/N$ we have
  \begin{displaymath}
    \langle \varphi_\ell, \varphi_k \rangle_N = \ee^{\ii(\ell-k)\theta_1} \frac{1}{2N}\frac{1-\ee^{2\pi\ii(k-\ell)}}{1-\ee^{\pi\ii(k-\ell)/N}}=0.
  \end{displaymath}
  This proves that $\{\varphi_n\}_{n=-N}^{N-1}$ forms an orthonormal basis with respect to the inner product $\langle \,\cdot\,,\, \cdot \,\rangle_N$. It follows that $\langle f, h \rangle = \langle f,h \rangle_N$ for all $f$ and $h$ in the span of $\{\varphi_n\}_{n=-N}^{N-1}$. Inserting $h(x) = \varphi_n(x)$ into the expression for the discrete inner product and then using equation \eqref{eqn:unitcircle} yields \eqref{eqn:discretecoeffs}. 
  \end{proof}

\subsection{Speed of convergence}

\begin{theorem}
  Let $f \in \mathrm{L}_2(\BB{R})$. The generalised Fourier coefficients satisfy 
  \begin{equation}
    \label{convergence}
    \langle f, \varphi_n \rangle =\O{\rho^{-|n|}},
  \end{equation}
  for some $\rho > 1$ if and only if the function $t \mapsto \left(1 - 2\ii t\right)\!f(t)$ can be analytically continued to the set
  \begin{displaymath}
    C_\rho = \overline{\BB{C}} \setminus \left(\BB{D}_{r_\rho}(a_\rho) \cup \BB{D}_{r_\rho}(\overline{a}_\rho)\right)
  \end{displaymath}
  where $\overline{\BB{C}}$ is the Riemann sphere consisting of the complex plane and the point at infinity, and $\BB{D}_r(a)$ is the disc with centre $a \in \BB{C}$ and radius $r > 0$, with
  \begin{displaymath}
    a_\rho = \frac{\ii}{2} \frac{\rho + \rho^{-1}}{\rho - \rho^{-1}},\qquad  r_\rho = \frac{1}{\rho - \rho^{-1}}.
  \end{displaymath}
\end{theorem}
\begin{proof}
  See \cite{weideman1995computing} and \cite{boyd1987spectral}.
\end{proof}

\begin{figure}[tb]
  \begin{center}
    \includegraphics[width=150pt]{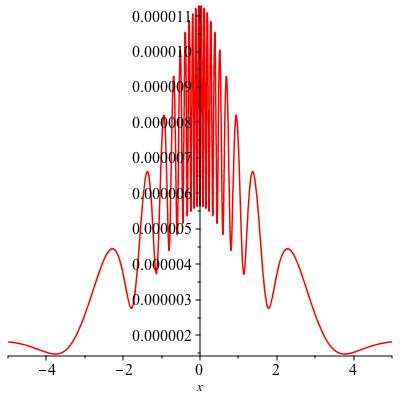}\hspace*{15pt}\includegraphics[width=150pt]{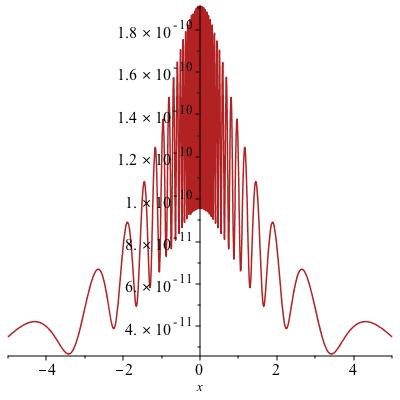}
    
    \includegraphics[width=150pt]{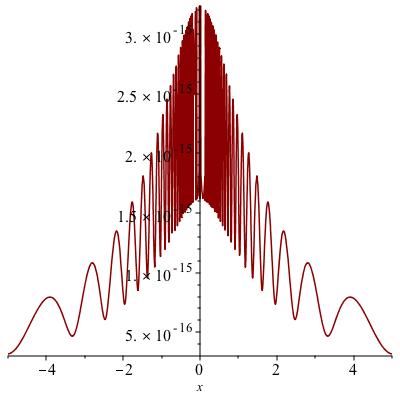}\hspace*{15pt}\includegraphics[width=150pt]{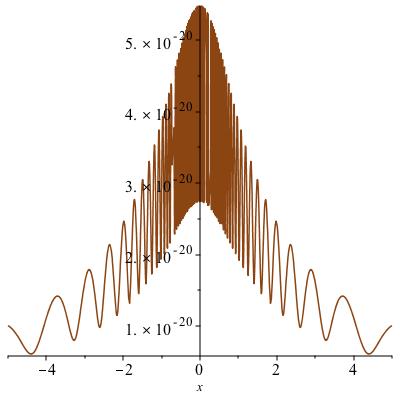}    
    \caption{MT errors for example \R{Example} with $N=10,20,30,40$.}
    \label{fig:3.2}
  \end{center}
\end{figure}

\begin{figure}[tb]
  \begin{center}
    \includegraphics[width=150pt]{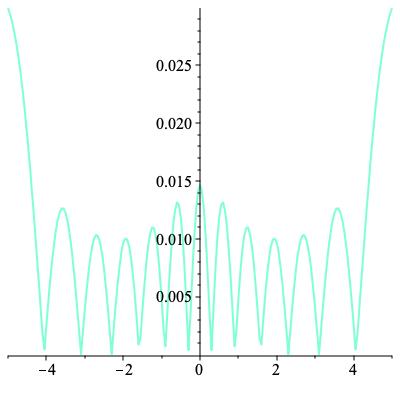}\hspace*{15pt}\includegraphics[width=150pt]{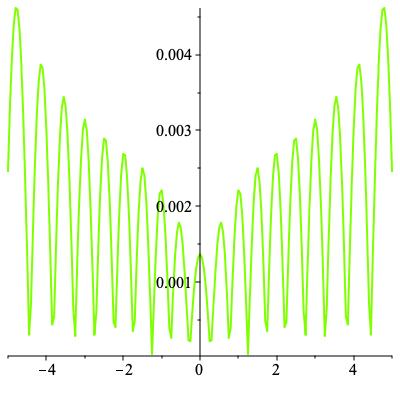}
    
    \includegraphics[width=150pt]{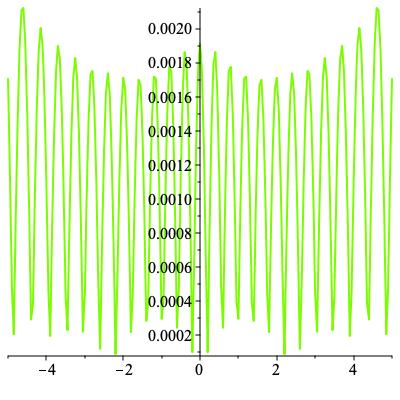}\hspace*{15pt}\includegraphics[width=150pt]{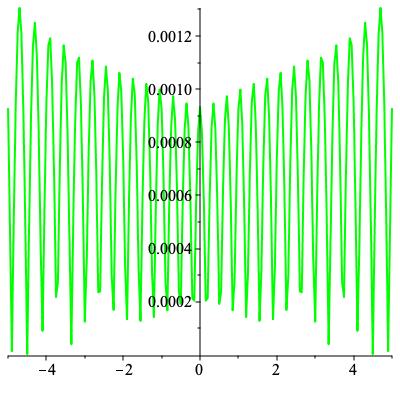}    
    \caption{Hermite function errors for example \R{Example} with $N=10,20,30,40$.}
    \label{fig:3.3}
  \end{center}
\end{figure}

As was noted by Weideman, for exponential convergence we require $f$ to be analytic at infinity, which is of meagre practical use. An example for a function $f$ of this kind is
\begin{equation}
  \label{Example}
  f(x)=\frac{1}{1+x^2}
\end{equation}
Since $f$ is a meromorphic function with singularities at $\pm\ii$, we obtain exponential decay with $\rho=3$ --  this is evident from the explicit expansion
\begin{displaymath}
  \frac{1}{1+x^2}=-\sqrt{2\pi}\sum_{n=-\infty}^{-1}\frac{(-\ii)^n}{3^{-n}} \varphi_n(x)+ \sqrt{2\pi}\sum_{n=0}^\infty \frac{(-\ii)^n}{3^{n+1}} \varphi_n(x),
\end{displaymath}
whose proof we leave to the reader. This is demonstrated in Fig.\ref{fig:3.2}, where we display the errors $\left|f(x)-\sum_{n=0}^N \hat{f}_n\varphi_n(x)\right|$ for $N=10,20,30$ and $40$. Compare this with Fig.~\ref{fig:3.3}, where we have displayed identical information for an expansion in Hermite functions. Evidently, MT errors decay at an exponential speed, while the error for Hermite functions decreases excruciatingly slowly as $N$ increases. 

\begin{figure}[tb]
  \begin{center}
    \includegraphics[width=150pt]{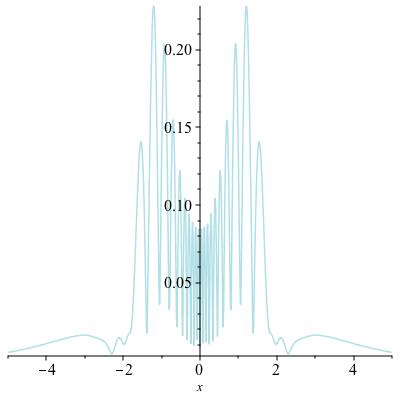}\hspace*{15pt}\includegraphics[width=150pt]{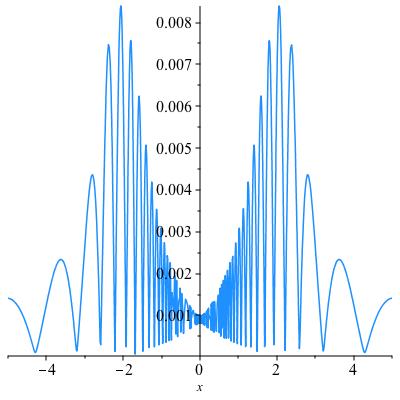}
    
    \includegraphics[width=150pt]{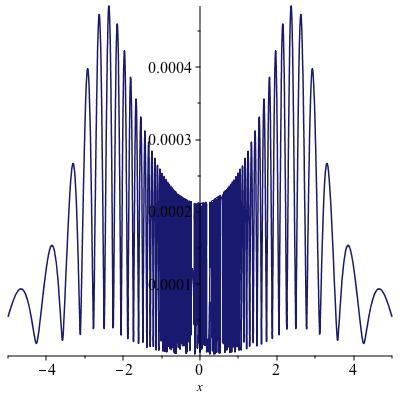}\hspace*{15pt}\includegraphics[width=150pt]{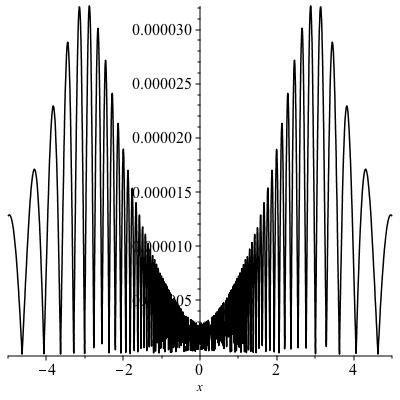}    
    \caption{MT errors for example \R{Fxample} with $N=10,40,70,100$.}
    \label{fig:3.4}
  \end{center}
\end{figure}

\begin{figure}[tb]
  \begin{center}
    \includegraphics[width=150pt]{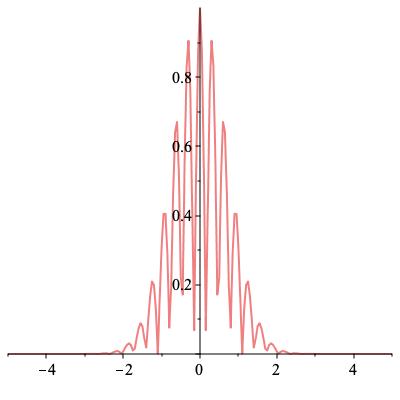}\hspace*{15pt}\includegraphics[width=150pt]{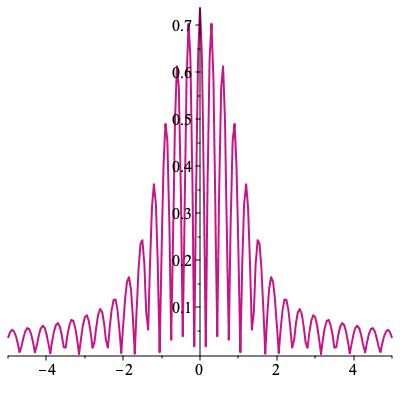}
    
    \includegraphics[width=150pt]{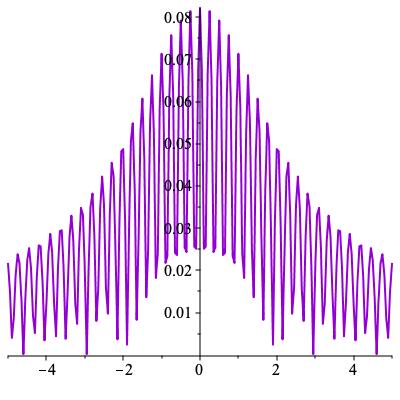}\hspace*{15pt}\includegraphics[width=150pt]{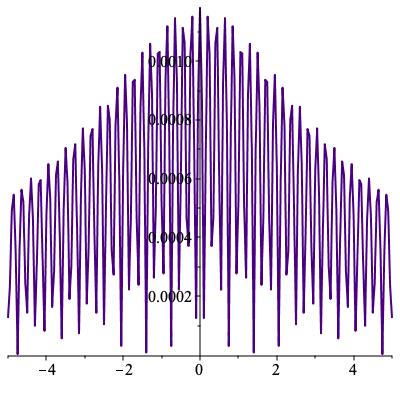}    
    \caption{Hermite function errors for example \R{Fxample} with $N=10,40,70,100$.}
    \label{fig:3.5}
  \end{center}
\end{figure}

\begin{table}[!htb]
  \centering
\caption{The rate of decay of the coefficients $\hat{f}_n$ in an MT approximation of different functions.}
\begin{tabular}{|c|c|}
  \hline & \\[-6pt]
  $\hspace*{12pt}f(x)\hspace*{12pt}$ & $\hat{f}_n$\\
   & \\[-6pt]
  \hline
  & \\[-6pt]
  $\displaystyle \frac{1}{1+x^4}$ & $\O{\rho^{-|n|}}$, $\rho=1+\sqrt{2}$\\[2pt]
    & \\[-6pt]
  $\ee^{-x^2}$ & $\O{\ee^{-3|n|^{2/3}/2}}$\\
    & \\[-6pt]
  $\CC{sech}\,x$ & $\O{\ee^{-2|n|^{1/2}}}$\\
    & \\[-6pt]
  $\displaystyle \frac{\sin x}{1+x^2}$ & $\O{|n|^{-5/4}}$\\
    & \\[-6pt]
  $\displaystyle \frac{\sin x}{1+x^4}$ & $\O{|n|^{-9/4}}$ \\[10pt]\hline
\end{tabular}
\label{tab:1}
\end{table}%

\begin{figure}[tb]
  \begin{center}
    \includegraphics[width=150pt]{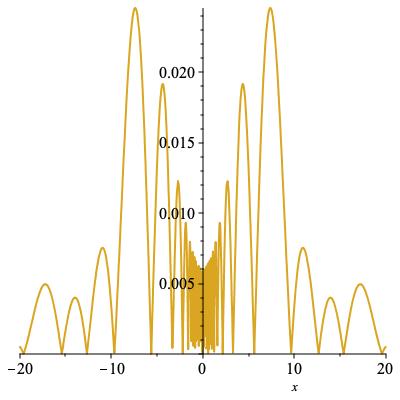}\hspace*{15pt}\includegraphics[width=150pt]{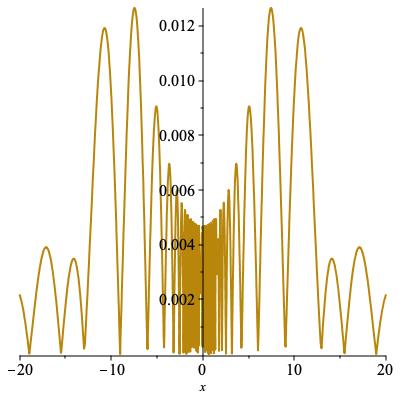}
    
    \includegraphics[width=150pt]{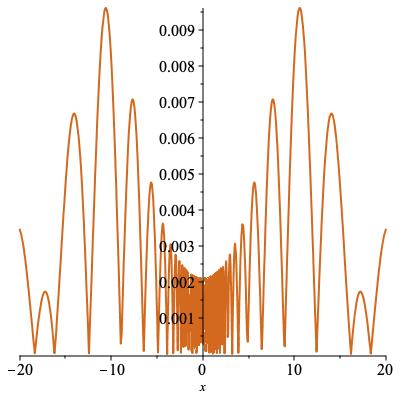}\hspace*{15pt}\includegraphics[width=150pt]{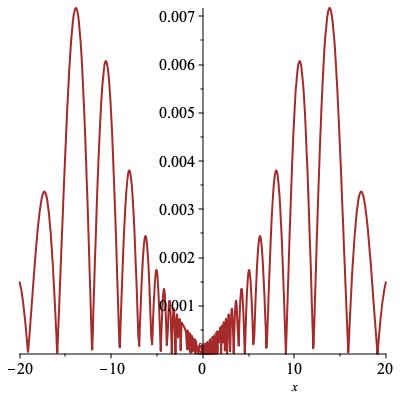}    
    \caption{MT errors for $f(x)=\sin x/(1+x^2)$ with $N=20,40,60,80$.}
    \label{fig:3.6}
  \end{center}
\end{figure}

Meromorphic functions, however, are hardly at the top of the agenda when it comes to spectral methods. In particular, in the case of dispersive hyperbolic equations we are interested in wave packets -- strongly localised functions, exhibiting double-exponential decay away from an envelope within which they oscillate rapidly. An example (with fairly mild oscillation) is the function
\begin{equation}
  \label{Fxample}
  f(x)=\ee^{-x^2}\cos(10 x).
\end{equation}
Since $f$ has an essential singularity at infinity, there is no $\rho>1$ so that \R{convergence} holds -- in other words, we cannot expect exponentially-fast convergence. We report errors for MT and Hermite functions in Figs~\ref{fig:3.4} and~\ref{fig:3.5} respectively for $N=10,40,70$ and $100$: definitely, the convergence of MT slows down (part of the reason is also the oscillation) but it still is  superior to Hermite functions. 

The general rate of decay of the error (equivalently, the rate of decay of $|\hat{f}_{|n|}|$ for $n\gg1$ for analytic functions and the MT system) is unknown, although \cite{weideman1995computing} reports interesting partial information, which we display in Table~\ref{tab:1} (taken from \cite{weideman1995computing}). The rate of decay does not seem to follow simple rules. For some functions the rate of decay is  spectral (faster than a reciprocal of any polynomial), yet sub-exponential. For  other functions it is polynomial (and fairly slow). Fig.~\ref{fig:3.6} exhibits MT errors for $f(x)=\sin x/(1+x^2)$ and $N=20,40,60,80$: evidently it is  in line with Table~\ref{tab:1}. It is fascinating that such a seemingly minor change to \R{Example} has such far-reaching impact on the rate of convergence. This definitely calls for further insight.

A future paper will address the rate of approximation of wave packets by both the MT basis and other approximation schemes.

\setcounter{equation}{0}
\setcounter{figure}{0}
\section{Characterisation of mapped and weighted Fourier bases}

The most pleasing feature of the MT basis is that the coefficients can be expressed as Fourier coefficients of a modified function. They can then be approximated using the Fast Fourier Transform. Are there other orthogonal systems like this? 

Let us consider all orthonormal systems $\Phi = \{\varphi_n\}_{n\in\BB{Z}}$ in $\CC{L}_2(\BB{R})$ with a tridiagonal skew-Hermitian differentiation matrix such that for all $f \in \CC{L}_2(\BB{R})$, the coefficients are equal to the classical Fourier coefficients of $k(\theta) f(h(\theta))$, $-\pi<\theta<\pi$, for some functions $k$ and $h$ (with a possible diagonal scaling by $\{\gamma_n\}_{n\in\BB{Z}}$). Specifically, we consider the {\em ansatz\/}
\begin{equation}
  \label{eqn:integralintheta}
  \langle f, \varphi_n \rangle = \gamma_n \int_{-\pi}^\pi \ee^{-\ii n \theta} k(\theta) f(h(\theta))  \, \mathrm{d}\theta,\qquad n\in\BB{Z}.
\end{equation} 
We assume that $h : (-\pi,\pi) \to \BB{R}$ is a differentiable function which is strictly increasing and onto, whose derivative is a measurable function. This implies the existence of a differentiable, strictly increasing inverse function $H : \BB{R} \to (-\pi,\pi)$. The chain rule implies $h'(\theta)H'(h(\theta)) \equiv 1$ (so that $H'$ is also a measurable function). The function $k$ is assumed to be a complex-valued $\CC{L}_2(-\pi,\pi)$ function (which makes the integral in \eqref{eqn:integralintheta} well defined). The constants $\gamma_n$ are complex numbers. We assume nothing more about $k$, $h$ and $\gamma_n$ in this section (but \emph{deduce\/} considerably more).

Making the change of variables $x = h(\theta)$ yields,
\begin{equation}
  \label{eqn:integralinx}
  \langle f, \varphi_n \rangle = \int_{-\infty}^{\infty} \gamma_n \ee^{-\ii n H(x)} k(H(x)) H'(x) f(x) \, \mathrm{d}x.
\end{equation}
For this to hold for all $f \in \CC{L}_2(\BB{R})$, we must have
\begin{equation}\label{eqn:FLform}
\varphi_n(x) = \overline{\gamma}_n K(x) \ee^{\ii n H(x)},
\end{equation}
where $K(x) = H'(x)\overline{k}(H(x))$.

How does this fit in with the MT basis? In the special case of Malmquist--Takenaka we have
\begin{eqnarray*}
  h(\theta) = \frac12 \mathrm{tan} \frac{\theta}{2}, &\qquad&  H(x) = 2\mathrm{tan}^{-1}(2x),\\
  k(\theta) = 1 - \ii\, \mathrm{tan} \frac{\theta}{2}, &\qquad&  K(x) = \sqrt{\frac{2}{\pi}}\frac{1}{1-2\ii x}, \\
  \gamma_n = (-\ii)^n .&&
  \end{eqnarray*}
 We prove the following theorem which characterises the Malmquist--Takenaka system as (essentially) the only one of the kind described by equation \eqref{eqn:FLform}.
\begin{theorem}
  \label{thm:characterize}
  All systems $\Phi = \{\varphi_n\}_{n\in\BB{Z}}$ of the form  \eqref{eqn:FLform}, such that
  \begin{enumerate}
    \item $\Phi$ is  orthonormal  in $\CC{L}_2(\BB{R})$,
    \item $\Phi$ has a tridiagonal skew-Hermitian differentiation matrix as in equation \eqref{eq:sHerm}, but indexed by all of $\BB{Z}$,
    \end{enumerate}
  are of the form
  \begin{equation}\label{eqn:generalFL}
    \varphi_n(x) = \gamma_n \sqrt{\frac{\left|\mathrm{Im}\lambda\right|}{\pi}} \ee^{\ii \omega x}  \frac{\left( \lambda-x \right)^{n+\delta}}{\left(\overline{\lambda}-x\right)^{n+\delta + 1}}
  \end{equation}
  where $\omega,\delta \in \BB{R}$, $\lambda \in \BB{C} \setminus \BB{R}$ and $\gamma_n \in \BB{C}$ such that $|\gamma_n| = 1$ for all $n \in \BB{Z}$. The differentiation matrix in the case where $\gamma_n = (-\ii)^n$, $\mathrm{Im}\lambda = \frac12$ and $\omega = 0$ has the terms
  \begin{equation}
    \label{eqn:bncndelta}
    b_n = n +  \delta + 1, \qquad c_n = 2(n+\delta) + 1, \qquad n \in \BB{Z}.
  \end{equation}
\end{theorem}

\begin{proof}
  Let us derive some necessary consequences of orthonormality of $\Phi$ by applying the change of variables $x = h(\theta)$ to the inner product.
  \begin{eqnarray}
  \int_{-\infty}^\infty \overline{\varphi}_n(x)\varphi_m(x) \, \mathrm{d} x &=& \gamma_n \overline{\gamma}_m\int_{-\infty}^\infty |K(x)|^2 \ee^{\ii(m-n)H(x)} \,\mathrm{d} x \\
  &=& \gamma_n \overline{\gamma}_m\int_{-\pi}^\pi h'(\theta) |K(h(\theta))|^2 \ee^{\ii(m-n)\theta} \, \mathrm{d}\theta.
  \end{eqnarray}
Orthogonality implies that the function $\theta \mapsto h'(\theta) |K(h(\theta))|^2$ is orthogonal to $\theta \mapsto \ee^{\ii k \theta}$ for all $k \in \BB{Z} \setminus \{0\}$. It is therefore a constant function. This constant is positive since $h$ is strictly increasing and $K$ is not identically zero. Normality of the basis implies that $|\gamma_n|^2 = \left[2\pi h'(\theta) |K(h(\theta))|^2 \right]^{-1}$, which is a constant independent of $n$. We can absorb this constant into $K$ and assume that $|\gamma_n| = 1$ for all $n \in \BB{Z}$. Therefore, $h'(\theta)|K(h(\theta))|^2 \equiv 1/(2\pi)$, which is equivalent to $|K(x)|^2 = H'(x)/(2\pi)$.

Since $\varphi_0(x) = \gamma_0 K(x)$ and $\varphi_0$ is infinitely differentiable (because it is proportional to the inverse Fourier transform of a superalgebraically decaying function $g$), we deduce that $K$ must be infinitely differentiable. The relationship $H'(x) = 2\pi |K(x)|^2$ therefore implies that $H$ is infinitely differentiable; in particular $H''(x) = 4\pi \Re\!\left[K'(x) \overline{K}(x)\right]$. Furthermore, there exists an infinitely differentiable function $\kappa : \BB{R} \to \BB{R}$ such that 
\begin{displaymath}
  K(x) = \ee^{\ii\kappa(x)}\sqrt{\frac{H'(x)}{2\pi}}.
\end{displaymath}

Let us derive more necessary consequences by taking into account the tridiagonal skew-Hermitian differentiation matrix. For all $n\in\BB{Z}$,
\begin{eqnarray*}
& & K'(x)\gamma_n\ee^{\ii n H(x)} + K(x) \gamma_n\ii n H'(x) \ee^{\ii n H(x)} \\
& & = -\overline{b}_{n-1}K(x) \gamma_{n-1} \ee^{\ii (n-1) H(x)} + \ii c_n \gamma_n K(x)\ee^{\ii n H(x)} + b_{n} K(x) \gamma_{n+1}\ee^{\ii (n+1) H(x)}.
\end{eqnarray*}
Note that $K'(x) = \left[\ii\kappa'(x) + \frac{H''(x)}{2H'(x)}\right] \!K(x)$, so dividing through by $K(x)\gamma_n\ii \ee^{\ii n H(x)}$ leads to
\begin{displaymath}
  \kappa'(x) = c_n - n H'(x) + \overline{\beta}_{n-1} \ee^{-\ii H(x)} + \beta_n \ee^{\ii H(x)} + \ii \frac{H''(x)}{2H'(x)},
\end{displaymath}
where $\beta_n = -\ii b_n \gamma_{n+1}/\gamma_n$ (here we use the fact that $\gamma_n^{-1} = \overline{\gamma}_n$). Without loss of generality, we can assume that $\beta_n \in\BB{R}$ for all $n$ because the symmetries discussed in subsection \ref{subsec:symmetries} allow us to choose $\{\gamma_n\}_{n\in\BB{Z}}$ (because they are all of the form $\ee^{\ii \kappa_n}$ for real numbers $\{\kappa_n\}_{n\in\BB{Z}}$).

Since $\kappa$ and $H$ are real-valued functions and $c_n$ and $\beta_n$ are real for all $n \in \BB{Z}$,  equating real and imaginary parts yields
\begin{eqnarray}
  \kappa'(x) &=& c_n - nH'(x) + (\beta_n + \beta_{n-1})\cos H(x) \label{eqn:firstrequirement} \\
           0 &=& (\beta_{n} - \beta_{n-1}) \sin H(x) + \frac{H''(x)}{2H'(x)} \label{eqn:secondrequirement}
\end{eqnarray}
It follows that $\beta_{n}-\beta_{n-1}$ is a constant which is independent of $n$, so we can write $a = 2(\beta_{n} - \beta_{n-1})$ for some real constant $a$ and equation \eqref{eqn:secondrequirement} becomes
\begin{displaymath}
  H''(x) = -a H'(x)\sin H(x),
\end{displaymath}
which, after integrating with respect to $x$, becomes
\begin{displaymath}
  H'(x) = a\cos H(x) + b
\end{displaymath}
for some real constant $b$. Since $H$ maps $\BB{R}$ onto $(-\pi,\pi)$ in a strictly increasing manner, by the monotone convergence theorem we must have $H'(\pm \infty) = 0$. Since $H'(\pm \infty) = a \cos(\pm \pi) + b$, we must have $a = b$. Therefore, using the formula $\cos\theta+1 = 2\cos^2(\theta/2)$, we obtain,
\begin{displaymath}
  \frac12 H'(x)\sec^2\frac{H(x)}{2}  = a.
\end{displaymath} 
Integrating with respect to $x$, we get
\begin{displaymath}
  \tan \frac{H(x)}{2}= a x + c
\end{displaymath}
for some real constant $c$. Hence there exist real constants $A$ and $B$ such that
\begin{equation}
H(x) = 2 \arctan(Ax+B).
\end{equation}
Note that necessarily $A \neq 0$. All that remains is to determine $K(x)$, which can be done by determining $\kappa(x)$. Taking $n=0$ in equation \eqref{eqn:firstrequirement} gives us
\begin{equation}
\kappa'(x) = c_0 + (\beta_0 + \beta_{-1})\cos(2\arctan(Ax+B)).
\end{equation}
The antiderivative of $\cos(2\arctan(Ax+B))$ is $2A^{-1}\arctan(2x) - x$, so there exist real constants $\omega$, $a$ and $b$ such that
\begin{displaymath}
  \kappa(x) = \omega x + 2a\arctan(Ax+B) + b.
\end{displaymath}
Whence we deduce that
\begin{equation}
K(x) = \sqrt{\frac{A}{\pi}}\ee^{\ii \omega x + \ii 2a\arctan(Ax+B) + \ii b}\frac{1}{\sqrt{1+(Ax+B)^2}}
\end{equation}
Using $\exp(\ii 2 \arctan(Ax+B)) = (1+\ii(Ax+B))/(1-\ii(Ax+B)) = -(\lambda - x)/(\overline{\lambda} - x)$, where $\lambda = A^{-1}(\ii-B) \in \BB{C}\setminus\BB{R}$, we deduce
\begin{eqnarray}
\varphi_n(x) &=& \sqrt{\frac{A}{\pi}}\ee^{\ii \omega x + \ii b}(-1)^{n+a}  \frac{1}{\sqrt{1+(Ax+B)^2}} \frac{\left(\lambda+x\right)^{n+a}}{\left(\overline{\lambda} - x\right)^{n+a}}\\
             &=& \gamma_n \sqrt{\frac{|\mathrm{Im}\lambda|}{\pi}} \ee^{\ii \omega x} \frac{\left(\lambda+x\right)^{n+a-\frac12}}{\left(\overline{\lambda} - x\right)^{n+a+\frac12}},
\end{eqnarray}
where $\gamma_n = \ee^{\ii b} (-1)^{n+a}$. Letting $\delta = a - \frac12$ shows that the system $\Phi$ must necessarily be of the form in equation \eqref{eqn:generalFL}. To complete the proof we must turn to the question of sufficiency. A derivation exactly as in subsection~3.2 but with $n$ replaced by $n+\delta$ verifies the explicit form of the coefficients \eqref{eqn:bncndelta} for the case $\gamma_n = (-\ii)^n$, $\lambda = \ii/2$ and $\omega = 0$. The symmetry considerations in subsection \ref{subsec:symmetries} show that the other values of $\gamma_n$, $\lambda$ and $\omega$ yield orthonormal systems with a tridiagonal skew-Hermitian differentiation matrix too.
\end{proof}

\section{Concluding remarks}

The subject matter of this paper is the theory of complex-valued orthonormal systems in $\CC{L}_2(\BB{R})$ with a tridiagonal, skew-Hermitian differentiation matrix. On the face of it, this is a fairly straightforward generalisation of the work of \cite{iserles18osssd}. Yet, the more general setting confers important advantages. In particular, it leads in a natural manner to the Malmquist--Takenaka system. The latter is an orthonormal system of rational functions, which we have obtained from Laguerre polynomials through the agency of the Fourier transform. The MT system has a number of advantages over, say, Hermite functions, which render it into a natural candidate for spectral methods for the discretization of differential equations on the real line. It allows for an easy calculus, because MT expansions can be straightforwardly multiplied. Most importantly, the calculation of the first $N$ expansion coefficients can be accomplished, using FFT, in $\O{N\log_2N}$ operations. Moreover, the MT system is essentially unique in having the latter feature. 

The FFT, however, is not the only route toward `fast' computation of coefficients in the context of orthonormal systems on $\CC{L}_2(\BB{R})$ with skew-Hermitian or skew-symmetric differentiation matrices. In  \cite{iserles19fao} we characterised all such real systems (thus, with a skew-symmetric differentiation matrix) whose coefficients can be computed with either Fast Cosine Transform, Fast Sine Transform or a combination of the two, again incurring an $\O{N\log_2N}$ cost. We prove there that there exist exactly four systems of this kind.

The connections laid out in Section 3 between the Fourier--Laguerre functions and the Szeg\H{o}--Askey polynomials (and hence Jacobi polynomials via the Delsarte--Genin transformation), are suggestive of a possible generalisation of Theorem \ref{thm:characterize} on the characterisation of the MT basis. It may be possible to characterise all systems which are orthonormal, have a tridiagonal skew-Hermitian differentiation matrix, and which are of the form
\begin{displaymath}
\varphi_n(x) = \Theta(x) \Pi_n\!\left(\ee^{\ii H(x)}\right)\!,
\end{displaymath}
where $\Theta \in \CC{L}_2(\BB{R})$, $H$ maps the real line onto $(-\pi,\pi)$, and $\{\Pi_n\}_{n\in\BB{Z}_+}$ is a system of orthogonal polynomials on the unit circle. The expansion coefficients for a function in such a basis are equal to expansion coefficients of a mapped and weighted function in the orthogonal polynomial basis $\{\Pi_n\}_{n\in\BB{Z}_+}$. The Fourier--Laguerre bases, in particular the MT basis, are certainly within this class of functions, but one can ask if there are more.

From a practical point of view, it is worth noticing that while the MT basis elements decay like $|x|^{-1}$ as $x \to \pm\infty$, the Fourier-Laguerre functions decay like $|x|^{-1-\alpha/2}$ where $\alpha > -1$ is the parameter in the generalised Laguerre polynomial. For the approximation of functions with a known asymptotic decay rate it may be advantageous to use a basis with the same decay rate.

The jury is out on which is the `best' orthonormal $\CC{L}_2(\BB{R})$ system with a skew-Hermitian (or skew-symmetric)  tridiagonal differentiation matrix and whose first $N$ coefficients can be computed in $\O{N\log_2N}$ operations. While some considerations have been highlighted in  \cite{iserles19fao}, probably the most important factor is the speed of convergence. Approximation theory in $\CC{L}_2(\BB{R})$ is poorly understood and much remains to be done to single out optimal orthonormal systems for different types of functions. Partial results, e.g.\ in \cite{ganzburg18eeb,weideman95tao}, indicate that the speed of convergence of such systems is a fairly delicate issue.

\section*{Acknowledgements}

The authors with to acknowledge helpful discussions with  Adhemar Bultheel (KU Leuven), Margit Pap (P\'ecs) and Andr\'e Weideman (Stellenbosch).

\bibliographystyle{agsm}
\bibliography{19019}

\end{document}